\documentclass[a4paper,twoside,10pt]{article}
\usepackage[a4paper,left=3cm,right=3cm, top=3cm, bottom=3cm]{geometry}
\usepackage{soul}
\usepackage{mathrsfs}
\usepackage{amsmath}
\usepackage{amsthm}
\usepackage{amssymb}
\usepackage{cite}
\usepackage{color}
\usepackage{cancel}
\usepackage{appendix}
\usepackage{subcaption}
\usepackage{tikz}
\usepackage{soul}
\usepackage[author={Lorenzo}]{pdfcomment}
\theoremstyle{plain}
\newtheorem{theorem}{Theorem}[section]

\newtheorem{lemma}[theorem]{Lemma}
\newtheorem{proposition}[theorem]{Proposition}
\theoremstyle{remark}
\newtheorem{remark}{Remark}

\newcommand{\eremk}{\hbox{}\hfill\rule{0.8ex}{0.8ex}}

\DeclareTextCompositeCommand{\u}{T1}{i}{\u\imath}

\newcommand{\Norm}[1]{{\left\|{#1} \right\|}}
\newcommand{\SemiNorm}[1]{{\left|{#1} \right|}}
\newcommand{\jump}[1]{\left[\!\left[#1\right]\!\right]}
\newcommand{\avg}[2][F]{\{#2\}_{#1}}
\newcommand{\Normth}[1]{{\left\vert\kern-0.25ex\left\vert\kern-0.25ex\left\vert #1 
\right\vert\kern-0.25ex\right\vert\kern-0.25ex\right\vert}}
\newcommand{\Normf}[1]{{\left\vert\kern-0.25ex\left\vert\kern-0.25ex\left\vert\kern-0.25ex\left\vert #1 
\right\vert\kern-0.25ex\right\vert\kern-0.25ex\right\vert\kern-0.25ex\right\vert}}

\newcommand{\h}{h}

\newcommand{\nablah}{\nabla_\h}
\newcommand{\T}{T}
\newcommand{\E}{K}
\newcommand{\Tauh}{\mathcal T_\h}
\newcommand{\p}{p}
\newcommand{\pstar}{\p^*}
\newcommand{\psharp}{\p^\sharp}
\newcommand{\pprime}{\p'}
\newcommand{\psharpprime}{(\psharp)'}
\newcommand{\pddagprime}{(\pddag)'}
\newcommand{\pddag}{\p^{\ddag}}
\newcommand{\pnatural}{\p^{\natural}}

\DeclareMathOperator{\W}{W}
\newcommand{\Brho}{B_\rho}
\newcommand{\hOmega}{\h_\Omega}

\newcommand{\GammaN}{\Gamma_N}
\newcommand{\GammaD}{\Gamma_D}
\newcommand{\CSob}[3]{C_{\text{Sob}}(#1,#2,#3)}

\newcommand{\hE}{\h_\E}
\newcommand{\Fcalh}{\mathcal F_h}
\newcommand{\FcalE}{\mathcal F^\E}
\newcommand{\FcalEI}{\mathcal F^{\E I}}
\newcommand{\FcalED}{\mathcal F^{\E D}}
\newcommand{\FcalhI}{\mathcal F_h^I}
\newcommand{\FcalhB}{\mathcal F_h^B}
\newcommand{\FcalhD}{\mathcal F_h^D}
\newcommand{\FcalhID}{\mathcal F_h^{ID}}

\newcommand{\F}{F}
\newcommand{\hF}{\h_\F}
\newcommand{\frakTE}{\mathfrak T_\E}
\newcommand{\FET}{\F_\E^\T}
\newcommand{\nbf}{\mathbf n}

\newcommand{\nbfF}{\mathbf n_\F}
\newcommand{\q}{q}
\newcommand{\qprime}{\q'}

\newcommand{\qflat}{\q^\flat}
\newcommand{\pflat}{\p^\flat}
\newcommand{\s}{s}
\newcommand{\sprime}{\s'}
\let\L\relax
\DeclareMathOperator{\L}{L}
\newcommand{\Lbf}{\mathbf L}
\newcommand{\Wbf}{\mathbf W}
\let\div\relax
\newcommand{\div}{\nabla\cdot}
\newcommand{\Rbb}{\mathbb R}
\newcommand{\Nbb}{\mathbb N}
\DeclareMathOperator{\BV}{BV}

\newcommand{\CBV}{C_{\rm BV}}

\newcommand{\htildeF}{\widetilde\h_\F}
\newcommand{\CSP}{C_{\rm SP}}
\newcommand{\CSPsharp}{C_{\rm SP}^\sharp}

\newcommand{\postar}{\p\ostar}
\newcommand{\ostar}{1^*}
\newcommand{\CS}{C_{\rm S}}
\newcommand{\CSsharp}{\CS^\sharp}
\newcommand{\CSflat}{\CS^\flat}
\newcommand{\CSddag}{\CS^\ddag}
\newcommand{\CSnatural}{\CS^\natural}
\newcommand{\pnaturalprime}{(\pnatural)'}
\renewcommand{\r}{r}

\DeclareMathOperator{\card}{card}

\newcommand{\gN}{g_N}
\newcommand{\uh}{u_\h}

\title{\footnotesize Trace inequalities for piecewise $\W^{1,\p}$ functions over general polytopic meshes}
\author{\footnotesize{Michele Botti\thanks{MOX, Department of Mathematics, Politecnico di Milano, 20133 Milano, Italy (michele.botti@polimi.it)},   
Lorenzo Mascotto\thanks{Department of Mathematics and Applications, University of Milano-Bicocca, 20125 Milan, Italy (lorenzo.mascotto@unimib.it); 
Faculty of Mathematics, University of Vienna, 1090 Vienna, Austria;
IMATI-CNR, 27100 Pavia, Italy}}}
\date{}

\begin{document}
\maketitle

\begin{abstract}
\noindent
Trace inequalities are crucial tools to derive the stability
of partial differential equations with inhomogeneous,
natural boundary conditions.
In the analysis of corresponding Galerkin methods,
they are also essential to show convergence
of sequences of discrete solutions
to the exact one for data with minimal regularity
under mesh refinements and/or degree of accuracy increase.
In nonconforming discretizations,
such as Crouzeix-Raviart and discontinuous Galerkin,
the trial and test spaces consists of functions
that are only piecewise continuous:
standard trace inequalities
cannot be used in this case.
In this work, we prove several trace inequalities
for piecewise $\W^{1,\p}$ functions.
Compared to analogous results already available in the literature,
our inequalities are established:
(\emph{i}) on fairly general polytopic meshes (with arbitrary number of
facets and arbitrarily small facets);
(\emph{ii})  without the need of finite dimensional arguments
(e.g., inverse estimates, approximation properties
of averaging operators);
(\emph{iii})  for different ranges
of maximal and nonmaximal Lebesgue indices.

\medskip\noindent
\textbf{AMS subject classification}:
46E35, 65N30.

\medskip\noindent
\textbf{Keywords}:
Trace inequalities;
piecewise $\W^{1,\p}$ functions;
nonconforming finite elements;
polytopic meshes;
spaces of bounded variation.
\end{abstract}

\begin{flushright}
{\scriptsize\textsl{\indent E bada, Pinocchio,\\
non fidarti mai troppo di chi sembra buono\\
e ricordati che c’è sempre qualcosa di buono\\
in chi ti sembra cattivo.}\\
Carlo Collodi}
\end{flushright}

\section{Introduction}

\paragraph*{Overture.}
Trace inequalities are prominent
in the analysis of numerical methods
for partial differential equations;
they are particularly relevant, e.g., in the proof
of well-posedness and convergence
in presence of inhomogeneous, natural boundary conditions.
While standard trace inequalities are typically
enough for conforming methods,
one needs more refined counterparts
for nonconforming schemes,
which employ piecewise regular functions.
Such inequalities take the rough form
($\p$ and $\q$ are Lebesgue indices discussed
in Section~\ref{subsection:Lebesgue-Sobolev-indices} below;
$\nablah$ denotes the broken gradient
with respect to a decomposition of the domain~$\Omega$
as in Section~\ref{subsection:meshes-broken} below;
$\SemiNorm{\cdot}_{J,\p,\q}$ is related to jump terms
in the sense of~\eqref{seminorm-broken} below)
\[
\Norm{u}_{\L^\q(\partial\Omega)}
\lesssim 
\Norm{\nablah u}_{\Lbf^\p(\Omega)}
+ \SemiNorm{u}_{J,\p,\q}.
\]

\paragraph{An example.}
With the notation as in~\eqref{eq:Omega} below,
we consider the~$\p$-Laplacian problem,~$\p$ in $(1,\infty)$,
\begin{equation} \label{pLaplacian-strong}
\begin{cases}
\text{find } u:\Omega\to \mathbb R \text{ such that}\\
-\div(\vert\nabla u\vert^{\p-2} \nabla u) = 0 
        & \text{ in } \Omega,\\
u=0     & \text{ on } \GammaD , \\
\vert\nabla u\vert^{\p-2} \nabla u \cdot \textbf{n}=\gN
        & \text{ on } \GammaN.
\end{cases}
\end{equation}
Given Lebesgue indices as in~\eqref{indices} below,
for $\gN$ in~$\L^{\psharpprime}(\GammaN)$,
a weak formulation of~\eqref{pLaplacian-strong} reads
\begin{equation} \label{pLaplacian-weak}
\begin{cases}
\text{find } u \in V:= W^{1,\p}_{\GammaD}(\Omega) \text{ such that}\\
(\vert\nabla u\vert^{\p-2} \nabla u, \nabla v)_{0,\Omega}
= (\gN, v)_{0,\GammaN}
\qquad\forall v \in V .
\end{cases}
\end{equation}
Taking $v=u$ in~\eqref{pLaplacian-weak},
H\"older's inequality entails
\[
\SemiNorm{u}_{\W^{1,\p}(\Omega)}^\p
\le \Norm{\gN}_{\L^{\psharpprime}(\GammaN)}
    \Norm{u}_{\L^{\psharp}(\GammaN)}.
\]
The Sobolev embedding
$\W^{\frac{1}{\pprime},\p}(\GammaN)
\hookrightarrow\L^{\psharp}(\GammaN)$,
the trace inequality $\W^{1,\p}(\Omega) \hookrightarrow\W^{\frac{1}{\pprime},\p}(\GammaN)$,
and a Poincar\'e inequality in $\W^{1,\p}_{\GammaD}(\Omega)$
(the overall constant being~$C$) give
\[
\SemiNorm{u}_{\W^{1,\p}(\Omega)}^\p
\le C \Norm{\gN}_{\L^{\psharpprime}(\GammaN)}
    \SemiNorm{u}_{\W^{1,\p}(\Omega)}
\qquad \Longrightarrow \qquad
\SemiNorm{u}_{\W^{1,\p}(\Omega)}
\le C^{\frac{1}{\p-1}} \Norm{\gN}_{\L^{\psharpprime}(\GammaN)}^{\frac{1}{\p-1}}.
\]
If we consider a $\W^{1,\p}$-conforming finite element
discretization of~\eqref{pLaplacian-weak},
see, e.g., \cite{Barrett-Liu:1993},
then we deduce analogously a bound of the form
(here $\uh$ is the conforming finite element solution
over a given simplicial/Cartesian mesh
and for an arbitrary polynomial degree)
\[
\SemiNorm{\uh}_{\W^{1,\p}(\Omega)}
\le C^{\frac{1}{\p-1}} \Norm{\gN}_{\L^{\psharpprime}(\GammaN)}^{\frac{1}{\p-1}}.
\]
This is the starting point for proving convergence of the finite element
method over sequences of meshes under minimal data regularity;
cf., e.g., \cite{Di-Pietro.Ern:2010}.
Convergence for nonconforming
(e.g., discontinuous Galerkin, Crouzeix-Raviart)
methods require more sophisticated trace inequalities
since these schemes are based on finite dimensional spaces
that are only piecewise $\W^{1,\p}$;
the constants appearing therein should depend
neither on the size of the elements in the mesh
nor on the type of discretization space that is used.

\paragraph*{State-of-the-art.}
For simplicial meshes,
Girault and Wheeler~\cite[Proposition~5]{Girault-Wheeler:2008}
proved a trace inequality for lowest order
Crouzeix--Raviart elements based on a
Scott--Zhang-type regularization argument;
they showed that traces of discrete functions
have the same Lebesgue regularity
as that appearing in the optimal conforming case.
Buffa and Ortner~\cite[Theorem~4.4]{Buffa-Ortner:2009}
generalized those results to the case of
arbitrary order piecewise polynomial spaces,
again with maximal Lebesgue regularity;
the main technical tool here is the stability
of a lifting operator, whose stability properties
require polynomial inverse inequalities.
Analogous results are given in \cite[Lemma~3.6]{Zhao-Chung-Park-Zhou:2021},
where the main technical tool in the proof
is the lowest order Crouzeix-Raviart interpolant.
Trace inequalities for broken Sobolev spaces
were extended to regular polytopic meshes
in \cite[Lemma B.24, eq. (B.58)]{Droniou-Eymard-Gallouet-Guichard-Herbin:2018} for lowest order
and \cite[Theorem 6.7]{DiPietro-Droniou:2020}
for general order broken polynomial spaces;
the Lebesgue indices of the boundary and jump terms
are the same as that of the broken gradient
and again inverse estimates and averaging arguments
are the lynchpin of the analysis therein.
In the recent work~\cite{Badia-Droniou-Tushar:2025},
novel discrete trace inequalities
for hybrid methods were proved,
which involve a mesh-dependent
$H^{\frac12}$-seminorm on the boundary.
All the references above are based on
finite dimensional arguments
and consider either simplicial meshes
or regular polytopic meshes.

\paragraph*{Main results and advances.}
The three main results of this work
are Theorems~\ref{theorem:sob-trace-pnatural},
\ref{theorem:sob-trace-psharp},
and~\ref{theorem:sob-trace-3} below;
each of them establishes trace inequalities
for piecewise $\W^{1,\p}$ functions
involving different types of Lebesgue indices;
their applicability in the analysis of nonconforming
methods is discussed in Remark~\ref{remark:where-to-use} below.
Their proofs are not based on finite dimensional arguments
(e.g., polynomial inverse estimates,
smoothing/averaging operators)
but only on direct estimates;
moreover, they hold true on fairly general meshes
(with, e.g., arbitrarily small facets,
arbitrary number of facets per element).
The constants appearing in these inequalities
are provided, which are:
(\emph{i}) explicit in terms of constants
of other elementary inequalities
(e.g., Sobolev-Poincar\'e inequalities for broken Sobolev spaces);
(\emph{ii}) certain ``regularity'' parameters of a given mesh
(cf. Section~\ref{subsection:meshes-broken});
(\emph{iii}) the involved Lebesgue indices.
The first dependence is essential in the proof of convergence
for the $p$- and $hp$-versions of nonconforming methods
for linear and nonlinear problems,
since inverse estimates typically result in bounds
involving the polynomial degree of the scheme;
it is also essential for nonpolynomial methods.
The latter dependence is crucial while performing
refinement/coarsening within an adaptive mesh refinement;
meshes with elements obtained by merging smaller simplicial elements
are fully covered by our theory.
Another advancement of this manuscript is that
we exhibit different types of trace inequalities
allowing for maximal and nonmaximal Lebesgue regularity.

\paragraph*{Structure of the paper.}
Standard Sobolev spaces and their broken versions
on certain classes of polytopic meshes,
and spaces of functions with bounded variation (BV)
are detailed in Section~\ref{section:preliminaries};
there, we also review related inequalities,
including standard trace inequalities,
Sobolev-Poincar\'e inequalities
for broken $\W^{1,\p}$ spaces,
and bounds on the BV norm in terms of broken Sobolev norms.
Section~\ref{section:trace} is concerned with the proof
of the novel trace inequalities for broken Sobolev spaces,
based on some novel technical results.

\section{Polytopic meshes, broken spaces, and related inequalities}
\label{section:preliminaries}
We consider an open, bounded, Lipschitz domain
\begin{equation}\label{eq:Omega}
\text{$\Omega \subset \Rbb^d$,
$d$ in~$\Nbb\setminus\{1\}$,
with boundary $\Gamma:=\partial\Omega$, which we split into
$\Gamma=\Gamma_{\rm D}\cup\Gamma_{\rm N}$,
$\vert\Gamma_{\rm D}\vert \ne \emptyset$.}
\end{equation}
We assume that~$\Omega$
\begin{itemize}
\item is either star-shaped with respect
to a ball $\Brho$ of radius~$\rho$;
the diameter of~$\Omega$ is~$\hOmega$;
\item or admits a shape-regular decomposition
into simplices;
\item or is the union of a finite number of star-shaped domains.
\end{itemize}
The regularity of~$\Omega$ plays a role only
in the proof of Lemma~\ref{lemma:broken-Sobolev-LP} below;
cf.~\cite{Botti-Mascotto:2025-B}.

This section is structured as follows:
in Section~\ref{subsection:Lebesgue-Sobolev-indices}
we set the notation related to Lebesgue and Sobolev spaces,
and introduce some Lebesgue indices;
in Section~\ref{subsection:meshes-broken}
we introduce a class of fairly general polytopic meshes,
broken Sobolev spaces, and jump operators;
corresponding broken Sobolev seminorms and norms, along
with their basic properties are detailed
in Section~\ref{subsection:semi-norms-technical};
the space of functions of bounded variation,
its norm, and the relation with broken
Sobolev norms are discussed in Section~\ref{subsection:BV}.

\subsection{Lebesgue and Sobolev spaces, and special Lebesgue indices}
\label{subsection:Lebesgue-Sobolev-indices}
For a generic subset~$X$ of~$\Omega$ with diameter~$\h_X$,
and for all~$\p$ in~$[1,\infty)$,
we consider the Lebesgue and Sobolev spaces~$\L^\p(X)$ and~$\W^{1,\p}(X)$,
which we endow with norm, and seminorm and norm, respectively,
\[
\Norm{v}_{\L^\p(X)}^\p
:= \int_X \vert v \vert^\p ,
\]
and
\[
\SemiNorm{v}_{\W^{1,\p}(X)}
:=   \Norm{\nabla v}_{\Lbf^\p(X)},
\qquad\qquad
\Norm{v}_{\W^{1,\p}(X)}^\p
:=  \h_X^{-p} \Norm{v}_{\L^\p(X)}^\p
    +  \SemiNorm{v}_{\W^{1,\p}(X)}^\p .
\]
Vector version Sobolev spaces,
norms, and seminorms are denoted by
$\Wbf^{1,\p}(X)$,
$\Norm{\cdot}_{\Wbf^{1,\p}(X)}$,
and $\SemiNorm{\cdot}_{\Wbf^{1,\p}(X)}$,
respectively.

We define the subspace of~$\W^{1,\p}(X)$ of functions
with zero trace over~$\GammaD$:
\[
\W^{1,\p}_{\GammaD}(X)
:= \{ v \in \W^{1,\p}(X) \mid v_{|\GammaD}=0 \} .
\]
We further recall the following Sobolev embeddings \cite[Sect.~2.3]{Ern-Guermond:2021}:
\begin{itemize}
\item if $\p <d$,
$\W^{1,\p}(X) \hookrightarrow \L^{q}(X)$
for all $q$ in $[\p,\frac{\p d}{d-\p}]$;
\item if $\p = d$,
$\W^{1,\p}(X) \hookrightarrow \L^{q}(X)$
for all $q$ in $[\p,\infty)$.
\end{itemize}
We spell out the generic Sobolev embedding bound:
for~$p$ and~$q$ as above,
there exists a positive constant~$\CSob{\q}{\p}{X}$
such that
\begin{equation} \label{Sobolev-embedding}
\Norm{v}_{\L^{q}(X)}
\le \CSob{q}{\p}{X} 
    \h_X^{\frac{d}{q}- \frac{d}{\p} + 1} \Norm{v}_{\W^{1,\p}(X)}
\qquad\qquad\qquad
\forall v \in \W^{1,\p}(X).
\end{equation}
Given an index~$\p$ in $[1,\infty)$, we define
\begin{equation} \label{indices}
\begin{split}
& \pprime :=  \frac{\p}{\p-1};
\qquad\qquad
\pstar := 
\begin{cases}
\frac{\p d}{d-\p}   & \text{if } \p < d  \\
\infty              & \text{otherwise};
\end{cases}
\qquad\qquad
\psharp := 
\begin{cases}
\frac{\p(d-1)}{d-\p}    & \text{if } \p < d \\
\infty                  & \text{otherwise};
\end{cases}\\
&  \pddag := 
 \frac{\p d -1}{d-1} ;
\qquad\quad\;\,
\pnatural :=
\p + \frac{\p-1}{(d-1)\pprime} ;
\qquad\qquad\qquad
\pflat:= \frac{\p d}{\p+d-1}.
\end{split}
\end{equation}
Observe
\begin{equation} \label{relation-indices}
1\le\pflat\le\p\le\pnatural\le\pddag\le\psharp\le\pstar,
\qquad
(\pflat)^\sharp = \p = (\psharp)^\flat,
\qquad
\begin{cases}
 \psharp\le\postar & \text{if } \p\le\frac{2d-1}{d}  \\
 \psharp\ge\postar & \text{if } \p\ge\frac{2d-1}{d}.
\end{cases}
\end{equation}
In Figure~\ref{fig:indices}, we report the behaviour
of the indices in~\eqref{indices} for~$\p$ varying
in $[1,6]$ with step $1/10$ in dimensions $d=2$ (left) and $d=3$ (right).

\begin{figure}[htbp]
\centering
\begin{subfigure}[b]{0.49\textwidth}
\centering
\includegraphics[width=\linewidth]{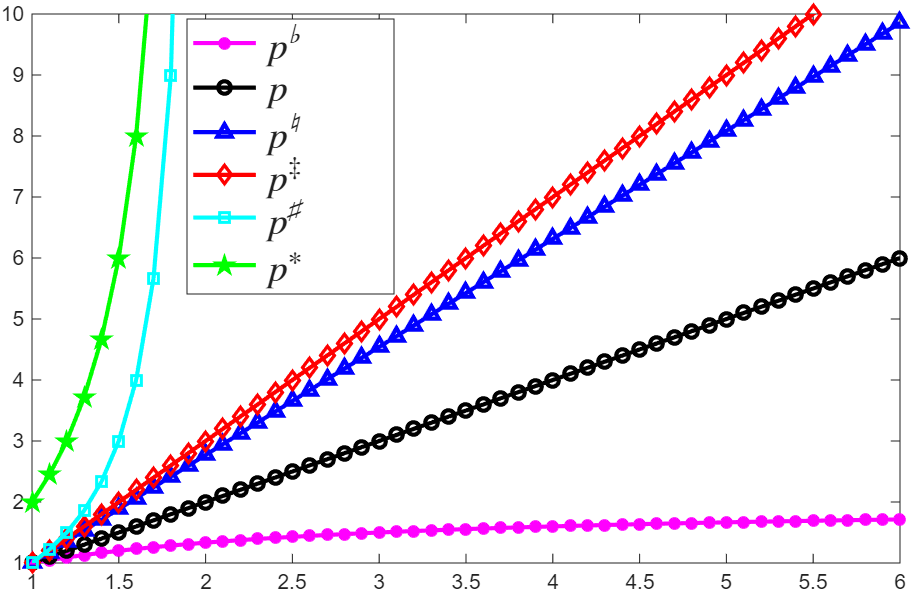}
\caption{$d=2$.}
\end{subfigure}
\hfill
\begin{subfigure}[b]{0.49\textwidth}
\centering
\includegraphics[width=\linewidth]{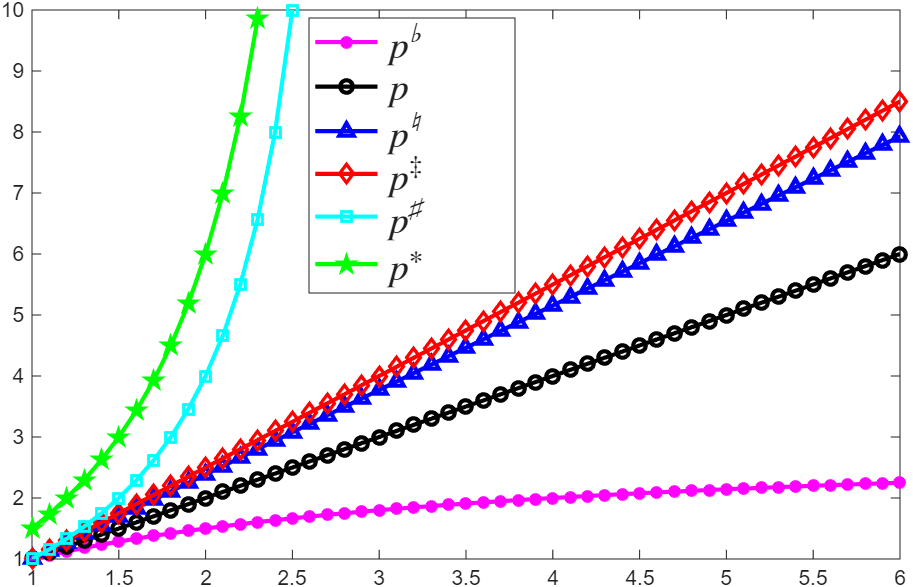}
\caption{$d=3$.}
\end{subfigure}
\caption{Behaviour of the indices in~\eqref{indices} for~$\p$ varying
in $[1,6]$ with step $1/10$.}
\label{fig:indices}
\end{figure}

\begin{remark}[Meaning of the Lebesgue indices]
For a given~$\p$ larger than~$1$:
$\pprime$ is its conjugate index;
$\pstar$ is the maximal Lebesgue regularity
of functions in $\W^{1,\p}(\Omega)$;
$\psharp$ is the maximal Lebesgue regularity
of the trace on $\partial\Omega$
of functions in $\W^{1,\p}(\Omega)$;
$\pddag$ is a technical index appearing
in Theorem~\ref{theorem:sob-trace-psharp} below;
$\pnatural$ is a technical index appearing
in Theorem~\ref{theorem:sob-trace-pnatural} below;
$\pflat$ is the ``dual'' of~$\psharp$.
\eremk
\end{remark}

\subsection{Meshes and broken Sobolev spaces} \label{subsection:meshes-broken}
We consider families of meshes~$\{\Tauh\}$,
where each~$\Tauh$ is a finite collection of disjoint, closed,
polytopic elements such that
$\overline{\Omega}=\bigcup_{\E\in\Tauh}\E$.
For each $\E$ in $\Tauh$, $\partial \E$ and $h_\E$ denote
the boundary and the diameter of~$\E$, respectively.
We associate each~$\Tauh$ with a set $\Fcalh$ covering the mesh skeleton,
i.e., $\bigcup_{\E\in\Tauh}\partial \E = \bigcup_{F\in\Fcalh} F$.
A facet $\F$ in $\Fcalh$ is a hyperplanar,
closed, and connected subset of $\overline{\Omega}$
with positive measure such that
\begin{itemize}
\item either there exist distinct $\E_{1,F}$ and $\E_{2,F}$ in $\Tauh$
such that $F\subseteq\partial \E_{1,F}\cap\partial \E_{2,F}$
and $\F$ is called an \emph{internal facet},
\item or there exists $\E_F$ in $\Tauh$ such that
$F\subseteq\partial \E_F\cap\partial\Omega$
and $\F$ is called a \emph{boundary facet}.
\end{itemize}
The notation $h_F$ is used for the diameter of the facet $F$ in $\Fcalh$.
Interfaces and boundary facets are collected in the subsets~$\FcalhI$ and~$\FcalhB$, respectively.
The set of facets of an element~$\E$
is~$\FcalE$; we also define
$\FcalED:= \FcalE\cap\FcalhD$ and $\FcalEI:=\FcalE\cap \FcalhI$.

Following \cite[Assumption 2.1]{Cangiani-Dong-Georgoulis:2016},
we demand that, for all~$\E$ in any~$\Tauh$,
there exists a partition $\mathfrak{T}_\E$
of~$\E$ into non-overlapping $d$-dimensional simplices.
Moreover, we assume that there exists a universal,
positive constant~$\gamma$ such that,
for all~$\E$ in any $\Tauh$ and
all $\T$ in $\frakTE$ with $\partial\T \cap \partial\E \neq \emptyset$,
given~$\FET$ the $(d-1)$-dimensional simplex
$\partial\T \cap \partial\E$,
\begin{equation} \label{regularity-mesh}
\gamma\hE \le d \vert \T \vert \vert \FET \vert^{-1}.
\end{equation}
Given
\begin{equation} \label{htildeF}
\htildeF
\quad \text{ either equal to } \hF
\text{ or equal to }\quad
\begin{cases}
    \min(\h_{\E_1} , \h_{\E_2}) & \text{if } \F\in\FcalhI,\; \F\subset\partial\E_1\cap\partial\E_2\\
    \h_\E & \text{if } \F\in\FcalhD,\; \F \in\FcalE,
\end{cases}         
\end{equation}
the following geometric bounds are a consequence
of~\eqref{regularity-mesh} and~\eqref{htildeF}
\begin{equation} \label{geometric-properties}
\begin{split}
\gamma \hE \vert\partial\E\vert
\le d  \vert\E\vert
\le d  \hE^d,
\qquad\qquad\qquad
\vert\F\vert
\le \hF^{d-1}
\le \htildeF^{d-1}
\le \hE^{d-1} .
\end{split}
\end{equation}
Families of shape-regular simplicial
and Cartesian meshes fall into the setting above.
Broken Sobolev spaces associated with $\Tauh$ are defined as 
\[
\W^{1,\p}(\Tauh):=
\left\{u\in \L^\p(\Omega)\,\mid\, {u}_{|\E}\in \W^{1,\p}(\E)
                        \qquad \forall \E\in\Tauh\right\}.
\]
We denote the piecewise gradient over~$\Tauh$ by~$\nablah$.
For every $v$ in~$\W^{1,\p}(\Tauh)$ and~$\F$ in~$\Fcalh$,
the jump operator on~$\F$ is given by
$$
\jump{v}_F:=
\begin{cases}
v_{|\E_{1,\F}} \nbf_{\E_{1,\F}} \cdot\nbfF
    + v_{|\E_{2,\F}} \nbf_{\E_{2,\F}} \cdot\nbfF 
    &\text{if } F \in \FcalhI,\; \F \subset \partial \E_{1,\F}\cap \partial \E_{2,\F}\\
v_{|\F} \nbf_{\E_{\F}} \cdot\nbfF
    &\text{if } F\in\FcalhB,\; \F \subset \partial\E_\F \cap \partial\Omega.
\end{cases}
$$
We omit the subscript $\F$ whenever it is clear
from the context.

\subsection{Broken seminorms and norms, and technical tools}
\label{subsection:semi-norms-technical}
We endow the space~$\W^{1,\p}(\Tauh)$ with the seminorms
\begin{subequations} \label{functionals-broken}
\begin{align}
\SemiNorm{v}_{\W^{1,\p;\q}(\Tauh)}
& := \Norm{\nablah v}_{\Lbf^\p(\Omega)}
  + \big( \sum_{\F\in\FcalhI} 
        \htildeF^{-\frac{\p}{\qprime}
        -\frac{d\p}{\pprime}
        +\frac{d\p}{\qprime}}
        \Norm{\jump{v}}^\p_{\L^{\q}(\F)}
        \big)^\frac1\p, \label{seminorm-broken}\\
\Norm{v}_{\W^{1,\p;\q}_{\GammaD}(\Tauh)}
&   := \Norm{\nablah v}_{\Lbf^\p(\Omega)}
    + \big( \sum_{\F\in\FcalhID} 
    \htildeF^{-\frac{\p}{\qprime}
    -\frac{d\p}{\pprime}
    +\frac{d\p}{\qprime}}
    \Norm{\jump{v}}^\p_{\L^{\q}(\F)}
    \big)^\frac1\p  . \label{norm-broken}
\end{align}
\end{subequations}

\begin{remark}[Comparison of broken seminorms] \label{remark:embedding-broken}
From definition~\eqref{norm-broken}
direct Lebesgue embeddings on facets,
and the second chain
of inequalities in~\eqref{geometric-properties},
we get
\begin{equation} \label{eq:trivial-bound}
\SemiNorm{v}_{\W^{1,\p;s}(\Tauh)}
\le \SemiNorm{v}_{\W^{1,\p;t}(\Tauh)}
\qquad\qquad\qquad
\forall t \in [1,\infty), \quad
\forall s \in [1,t].
\end{equation}
Moreover, proceeding as in \cite[Lemma 5.1]{DiPietro-Ern:2012}
and using the first inequality
in~\eqref{geometric-properties},
there also exists a positive constant~$C_{DG}$
only depending on~$\gamma$ and~$\Omega$ such that
\[
\SemiNorm{v}_{\W^{1,\p;\p}(\Tauh)}
\le C_{DG} \SemiNorm{v}_{\W^{1,\q;\q}(\Tauh)}
\qquad\qquad\qquad
\forall \q \in [1,\infty),\quad
\forall \p \in [1, \q].
\]
\eremk
\end{remark}


The seminorm in~\eqref{norm-broken}
is a norm, as a consequence of the next result,
which was proven in~\cite{Botti-Mascotto:2025-B}.

\begin{lemma}[Sobolev-Poincar\'e inequalities for piecewise $\W^{1,\p}$ functions] \label{lemma:broken-Sobolev-LP}
Let~$\{\Tauh\}$ be a family of meshes
as in Section~\ref{subsection:meshes-broken}
and~$\p$ be in $[1,\infty)$.
There exist positive constant~$\CSP$ and $\CSPsharp$
may possibly depend on~$\p$, $\GammaD$,
$\Omega$, $d$, and~$\gamma$ such that
\begin{equation} \label{eq:broken-Sobolev}
\Norm{v}_{\L^{\postar}(\Omega)}
\le \CSP \Norm{v}_{\W^{1,\p;\p}_{\GammaD}(\Tauh)}
\qquad\qquad\qquad
\forall v\in \W^{1,\p}(\Tauh),
\end{equation}
and, if~$\p$ further belongs to $[1,d)$,
\begin{equation} \label{eq:broken-Sobolev-sharp}
\Norm{v}_{\L^{\pstar}(\Omega)}
\le \CSPsharp \Norm{v}_{\W^{1,\p;\psharp}_{\GammaD}(\Tauh)}
\qquad\qquad\qquad
\forall v\in \W^{1,\p}(\Tauh).
\end{equation}
\end{lemma}
For the explicit dependence of~$\CSP$ and~$\CSPsharp$
on the parameters highlighted in the statement
of Lemma~\ref{lemma:broken-Sobolev-LP},
we refer to \cite[Theorems 1.5-1.6]{Botti-Mascotto:2025-B}.

We report the following Sobolev-trace inequalities
from \cite[Theorem~1.1 and eq. (26)]{Botti-Mascotto:2025-B}.
\begin{lemma}[Local Sobolev--trace inequality]
\label{lemma:trace-general}
Let~$\{\Tauh\}$ be a family of meshes
as in Section~\ref{subsection:meshes-broken},
$\p$ be in $[1,\infty)$,
and $\q$ be in $[\p,\infty)$.
Then, we have
\begin{equation} \label{intermediate-trace}
\Norm{v}_{\L^\q(\partial\E)}^\q
\le \frac{d}{\gamma \hE} 
    \vert \E\vert^{1-\frac{\qprime}{\pprime}}
    \Norm{v}^\q_{\L^{\pprime(\q-1)} (\E)}
    + \frac{\q}{\gamma}
        \Norm{\nabla v}_{\Lbf^\p(\E)}
        \Norm{v}_{\L^{\pprime(\q-1)}(\E)}^{\q-1}.
\end{equation}
\end{lemma}

We conclude this section with an auxiliary result.
\begin{lemma}[Auxiliary result] \label{lem:basic_sob1}
Given a mesh~$\Tauh$ in a sequence
as in Section~\ref{subsection:meshes-broken},
let $p$ be in~$[1,\infty)$ and $v$ in $\W^{1,\p}(\Tauh)$.
Then, $|v|^\p$ belongs to~$\W^{1,1}(\Tauh)$;
$|v|^{\q}$ also belongs to~$\W^{1,1}(\Tauh)$
for all $\q$ in $[\p,\psharp]$
if~$\p$ is in~$[1,d)$.
\end{lemma}
\begin{proof}
Given $v$ in $\W^{1,\p}(\Tauh)$, $|v|^\p$ belongs to $\L^1(\Omega)$
since $(|v|^\p)_{|\E}$ is in $\L^1(\E)$ for all $\E$ in~$\Tauh$. 
Using H\"older's inequality ($\p$ and $\pprime$)
and
$|\nabla (|v|^\p)_{|\E}|= p |v_{|\E}|^{p-1} |\nabla v_{|\E}|$
for all $\E$ in~$\Tauh$,
we deduce that $|v|^\p$ is in $\W^{1,1}(\Tauh)$.
In fact,
$$
\int_\E |\nabla (|v|^\p)|
\le \p \Big(\int_\E |v|^{p}\Big)^{\frac{p-1}p} 
    \Big(\int_\E |\nabla v|^{p}\Big)^{\frac1p}
\le \p \Norm{v}^{p-1}_{\L^{\p}(\E)}
    \SemiNorm{v}_{\W^{1,\p}(\E)}
\le \p \Norm{v}^{p}_{\W^{1,\p}(\E)}
\qquad \forall \E\in\Tauh.
$$
In order to show that $|v|^{\q}$ belongs
to~$\W^{1,1}(\Tauh)$ if~$\p$ is in $[1,d)$
and $\q$ is in~$[\p,\psharp]$,
we first remark that $|v|^{\psharp}$ belongs to~$\L^1(\Omega)$
using that $\psharp<p^*$ and
the corresponding Sobolev embedding~\eqref{Sobolev-embedding}.
Applying H\"older's inequality
and observing that $(\psharp -1)p'=p^*$,
we infer
$$
\int_\E |\nabla (|v|^{\psharp})|
= \psharp \int_\E |v_{|\E}|^{{\psharp}-1} |\nabla v_{|\E}|
\le \psharp \Big(\int_\E |v|^{p^*}\Big)^{\frac1{\pprime}}
    \Big(\int_\E |\nabla v|^{p}\Big)^{\frac1\p}
\le \psharp \Norm{v}^{\psharp-1}_{\L^{\pstar}(\E)}
    \Norm{v}_{\W^{1,\p}(\E)}.
$$
The Sobolev embedding in~\eqref{Sobolev-embedding} gives
$\Norm{v}_{\L^{\pstar}(\E)} \le \CSob{\pstar}{\p}{\E}  \Norm{v}_{\W^{1,\p}(\E)}$,
whence the right-hand side above is finite
and thus the assertion follows.
\end{proof}

\subsection{Functions of bounded variation}
\label{subsection:BV}
The space $\BV(\Omega)$ consists of functions
whose distributional derivative is a Radon measure
with finite total variation, i.e.,
$$
\BV(\Omega):= \{v\in \L^1(\Omega) \mid |D v|(\Omega)<\infty\}
\qquad\text{where}\qquad
|D v|(\Omega)
:= \sup_{\boldsymbol{\xi}\in \boldsymbol{C}^1_0(\Omega) \, ,  \,
\Norm{\boldsymbol{\xi}}_{\L^\infty(\Omega)}  \le1}
\int_\Omega v \ \nabla\cdot\boldsymbol{\xi},
$$
which we endow with the norm
\[
\Norm{v}_{\BV,\Omega}
:=  \Norm{v}_{\L^1(\Omega)} + |D v|(\Omega) .
\]
In what follows, we review some properties of functions of bounded variation;
cf. \cite{Ambrosio-Fusco-Pallara:2000} for a detailed presentation of these spaces.
The following result provides
the existence of a bounded, linear trace operator
from $\BV(\Omega)$ into $\L^1(\partial\Omega)$;
details can be found
in \cite[Theorem 5.5]{Lahti-Shanmugalingam:2018},
\cite[Theorem 4.2]{Buffa-Ortner:2009},
and \cite[Chapter~3]{Ambrosio-Fusco-Pallara:2000}.

\begin{lemma}
[Trace inequality in $\BV(\Omega)$] \label{lem:trace_sob_BV}
There exists a positive constant $\CBV$ only depending on~$\Omega$
such that
\begin{equation}\label{eq:trace_sob_BV}
\Norm{v}_{\L^1(\partial\Omega)}
\le \CBV \Norm{v}_{\BV,\Omega}
\qquad\qquad\qquad \forall v \in \BV(\Omega).
\end{equation}
\end{lemma}

Another important result states that
$\BV(\Omega)$ embeds continuously in broken Sobolev spaces
for any mesh as in Section~\ref{subsection:meshes-broken};
the proof is rather standard but we report it for the
sake of completeness.
\begin{lemma}[Bound of the $\BV$-norm
by broken Sobolev norms]\label{lem:bnd_BV}
Given $\CSP$ as in~\eqref{eq:broken-Sobolev}
and a polytopic mesh~$\Tauh$
as in Section~\ref{subsection:meshes-broken},
we have
$$
\Norm{v}_{\BV,\Omega}
\le (1+\vert\Omega\vert^\frac1 d \CSP)
    \Norm{v}_{\W^{1,1}_{\GammaD}(\Tauh)}
\qquad\qquad\qquad \forall v \in \W^{1,1}(\Tauh).
$$
\end{lemma}
\begin{proof}
Let $v$ be in $\W^{1,1}(\Tauh)$.
We first bound the total variation~$|D v|(\Omega)$
using an integration by parts
and H\"older's inequality ($1$ and $\infty$):
$$
\begin{aligned}
|D v|(\Omega)
&= \sup_{\boldsymbol\xi\in\boldsymbol{C}^1_0(\Omega)\, , \,
    \Norm{\boldsymbol\xi}_{\L^\infty(\Omega)}\le1} 
    \sum_{\E\in\Tauh}\int_\E v\ \div \boldsymbol{\xi} \\
& =\sup_{\boldsymbol\xi\in\boldsymbol{C}^1_0(\Omega)\, , \,
    \Norm{\boldsymbol\xi}_{\L^\infty(\Omega)}\le1}
    \Big( -\int_\Omega \nablah v\cdot\boldsymbol{\xi} 
    +\sum_{F\in\FcalhI}\int_F \jump{v}\,
    \boldsymbol{\xi}\cdot\boldsymbol{n}_F \Big) \\
& \le \sup_{\boldsymbol\xi\in\boldsymbol{C}^1_0(\Omega)\, , \,
    \Norm{\boldsymbol\xi}_{\L^\infty(\Omega)}\le1}
    \Big(\Norm{\nabla_h v}_{\Lbf^1(\Omega)}
    + \sum_{F\in\FcalhI} \Norm{\jump{v}}_{\L^1(F)} \Big) 
    \Norm{\boldsymbol\xi}_{\L^\infty(\Omega)}
\le \SemiNorm{v}_{\W^{1,1}(\Tauh)}.
\end{aligned}  
$$
Using H\"older's inequality ($\ostar$ and $(\ostar)'=d$)
and~\eqref{eq:broken-Sobolev}, we deduce
\[
\Norm{v}_{\L^{1}(\Omega)}
\le \vert\Omega\vert^\frac1 d \Norm{v}_{\L^{\ostar}(\Omega)}
\le \vert\Omega\vert^\frac1 d \CSP \Norm{v}_{\W^{1,1}_{\GammaD}(\Tauh)}.
\]
The assertion follows noting
\[
\Norm{v}_{\BV,\Omega} 
\le \SemiNorm{v}_{\W^{1,1}(\Tauh)}
    + \vert\Omega\vert^\frac1 d \CSP \Norm{v}_{\W^{1,1}_{\GammaD}(\Tauh)}.
\]
\end{proof}

\section{Trace inequalities in broken Sobolev spaces}\label{section:trace}
The aim of this section is to establish
several trace inequalities
in broken Sobolev spaces~$\W^{1,\p}(\Tauh)$.
Some technical, preliminary results are detailed in
Section~\ref{subsection:preliminary-results};
Section~\ref{subsection:trace-theorems} is instead
concerned with the proof of these trace inequalities
and a discussion on their applicability in the analysis
of certain nonconforming methods.

\subsection{Preliminary results} \label{subsection:preliminary-results}
We begin with an algebraic bound.
\begin{lemma}[Algebraic bound] \label{lemma:technical-ineq}
Given~$a \ge b\ge 0$ and $r\ge1$, we have
\[
a^r-b^r
\le r ( a-b ) 
     \frac{(a^{r-1}+b^{r-1})}{2}.
\]
\end{lemma}
\begin{proof}
If $b=0$, then the assertion is trivial.
Otherwise, we have
\[
a^r-b^r
= (a-b)(a^{r-1} + a^{r-2}b + \dots + ab^{r-2} + b^{r-1}).
\]
In order to conclude, it suffices
to observe that the function $f:\Rbb^+\to\Rbb^+$ defined such that
\[
f(t) = \frac{t^{r-1}+t^{r-2}+\dots+t+1}
        {t^{r-1}+1}
\]
attains its maximum in~$1$ and point out $f(1)=r/2$.
\end{proof}

We prove next an ancillary result.
\begin{proposition}[Preliminary Sobolev-trace inequalities for piecewise $\W^{1,\p}$ functions]
\label{prop:sob-trace-general}
Let~$\p$ be in~$[1,\infty)$,
$q$ in~$[\p,\psharp]\cap[p,\infty)$,
and $\s$ in~$[p,q]$.
Then, for all~$v$ in~$\W^{1,\p}_{\GammaD}(\Tauh)$,
we have
$$
\Norm{v}^{\q}_{\L^{\q}(\partial\Omega)} \le 
  C_1(\q,\Omega)\left[
    \Norm{v}^{\q-1}_{\L^{\pprime(\q-1)}(\Omega)}
    + \Big(\sum_{\E\in\Tauh}
    \hE^{d - \frac{(d-1)\pprime}{\sprime}}
    \Norm{v}^{\pprime(\q-1)}_{\L^{\sprime(\q-1)}(\partial\E)}
    \Big)^{\frac{1}{\pprime}}\right]
    \Norm{v}_{\W^{1,\p;\s}_{\GammaD}(\Tauh)},
$$
where, given~$\CSP$ in~\eqref{eq:broken-Sobolev}
and~$\CBV$ in~\eqref{eq:trace_sob_BV},
the constant $C_1$ is defined as
\begin{equation} \label{C1}
C_1(\q,\Omega) := 
\q \,\CBV (1+\vert\Omega\vert^\frac1 d \CSP).
\end{equation}
\end{proposition}
\begin{proof} 
Since $\q$ belongs to $[\p,\psharp]\cap[p,\infty)$,
Lemmas~\ref{lem:basic_sob1} and~\ref{lem:bnd_BV}
imply that $|v|^{\q}$
is in~$\W^{1,1}(\Tauh) \subset \BV(\Omega)$
for any~$v$ in $\W^{1,\p}(\Tauh)$.
Thus, using Proposition~\ref{lem:trace_sob_BV}
and Lemma~\ref{lem:bnd_BV}, we infer
\begin{equation}\label{eq:sob_start}
\begin{aligned}
\Norm{v}_{\L^{\q}(\partial\Omega)}^{\q}
& = \Norm{|v|^{\q}}_{\L^1(\partial\Omega)}
    \overset{\eqref{eq:trace_sob_BV}}{\le}
    \CBV\Norm{|v|^{\q}}_{\BV,\Omega}
    \overset{\eqref{eq:broken-Sobolev}}{\le}
    \CBV(1+\vert\Omega\vert^\frac1 d \CSP)
        \Norm{\, |v|^{\q}}_{\W^{1,1}_{\GammaD}(\Tauh)} \\
&   \overset{\eqref{seminorm-broken}}{=}
    \CBV (1+\vert\Omega\vert^\frac1 d \CSP)
    \Big(\sum_{\E\in\Tauh}\int_\E |\nabla(|v|^{\q})| 
    + \sum_{F\in\FcalhID}\int_F |\jump{|v|^{\q}}|\Big).
\end{aligned}
\end{equation}
For all $\E$ in $\Tauh$,
$|\nabla (|v_{|\E}|^{\q})|
= \q\ |v_{|\E}|^{\q-1} \ |\nabla v_{|\E}|$.
Lemma~\ref{lemma:technical-ineq} and the triangle inequality entail
\[
\jump{|v|^{\q}}\le \q \avg{|v|^{\q-1}}\jump{|v|} 
\le \q \avg{|v|^{\q-1}}|\jump{v}| 
\qquad\qquad\forall F\in\FcalhI. 
\]
Given~$C_1$ as in~\eqref{C1},
this allows us to write
\begin{equation} \label{eq:inequality-parameter-r}
 \Norm{v}^{\q}_{\L^{\q}(\partial\Omega)}
    \le  C_1 \Big(\sum_{\E\in\Tauh}\int_\E |v|^{\q-1}|\nabla v|
    + \sum_{F\in\FcalhID} 
    \int_F \avg{|v|^{\q-1}}|\jump{v}|
    \Big) = C_1 \big(\mathcal{I}_1 + \mathcal{I}_2).
\end{equation}
H\"older's inequality ($\p$ and~$\pprime$)
and H\"older's inequality for sequences ($\p$ and~$\pprime$)
give
\[
\begin{aligned}
\mathcal{I}_1 &\le \sum_{\E\in\Tauh}
        \Norm{v}^{\q-1}_{\L^{\pprime(\q-1)}(\E)}
        \Norm{\nabla v}_{\Lbf^\p(\E)}
    &\le  \Norm{v}^{\q-1}_{\L^{\pprime(\q-1)}(\Omega)}
    \Norm{\nablah v}_{\Lbf^\p(\Omega)}.
\end{aligned}        
\]
As for the term~$\mathcal{I}_2$,
we let $s$ be in~$[\p,\q]$ and define 
\begin{equation} \label{choice-alpha}
\alpha 
:= \frac{1}{\sprime} + \frac{d}{\pprime} - \frac{d}{\sprime}
>0.
\end{equation}
We use H\"older's inequality ($\s$ and~$\sprime$),
H\"older's inequality for sequences ($\p$ and~$\pprime$),
and the fact that $\htildeF$ in~\eqref{htildeF} is smaller than~$\hE$ for all~$\F$ in~$\FcalE$,
and get
\[
\begin{aligned}
\mathcal{I}_2
&   \le \sum_{\E\in\Tauh}
    \Big(\frac12 \sum_{F\in\FcalEI}
    \Norm{v_{|\E}}^{\q-1}_{\L^{\sprime(\q-1)}(\F)}
    \Norm{\jump{v}}_{\L^{\s}(\F)}
    +\sum_{F\in\FcalED}
    \Norm{v_{|\E}}^{\q-1}_{\L^{\sprime(\q-1)}(\F)}
    \Norm{\jump{v}}_{\L^{\s}(\F)}  \Big)\\
& \le \Big(\sum_{\E\in\Tauh}
    \sum_{F\in\Fcalh^\E}
    \htildeF^{\alpha\pprime} \Norm{v_{|\E}}^{\pprime(\q-1)}_{\L^{\sprime(\q-1)}(\F)}\Big)^{\frac{1}{\pprime}}
    \Big(\sum_{F\in\FcalhI}
    (2\htildeF)^{-\alpha\p}\Norm{\jump{v}}^\p_{\L^{\s}(\F)}
    + \sum_{F\in\FcalhD} 
    \htildeF^{-\alpha\p}
    \Norm{\jump{v}}^\p_{\L^{\s}(\F)}\Big)^{\frac1p} \\ 
& \le \Big(\sum_{\E\in\Tauh}
    \hE^{\alpha\pprime}
    \Norm{v}^{\pprime(\q-1)}_{\L^{\sprime(\q-1)}(\partial\E)}
    \Big)^{\frac{1}{\pprime}}
    \Big(\sum_{F\in\FcalhID} 
    \htildeF^{-\alpha\p}
    \Norm{\jump{v}}^\p_{\L^{\s}(\F)}\Big)^{\frac1p}\\ 
& \le \Big(\sum_{\E\in\Tauh}
    \hE^{\alpha\pprime}
    \Norm{v}^{\pprime(\q-1)}_{\L^{\sprime(\q-1)}(\partial\E)}
    \Big)^{\frac{1}{\pprime}}
    \Norm{v}_{\W^{1,\p;\s}_{\GammaD}(\Tauh)}.
\end{aligned}
\]
The conclusion follows by plugging
the bounds on~$\mathcal{I}_1$ and~$\mathcal{I}_2$
in~\eqref{eq:inequality-parameter-r}. 
\end{proof}

\subsection{Trace theorems} \label{subsection:trace-theorems}
We are now in a position to prove the main results of this work.
We consider three cases
and postpone their applicability in the analysis
of nonconforming methods
to Remark~\ref{remark:where-to-use} below:
\begin{itemize}
\item \textbf{case 1}:
$\p$ in $[1,\infty)$ and $\q$ in $[1,\pnatural]$;
\item \textbf{case 2}: 
   ~$\p$ in $[1,d)$ and $\q$ in $[\p, \psharp]$;
\item \textbf{case 3}:
   ~$\p$ in $[d,\infty)$ and $\q$ in $[\p,\infty)$.
\end{itemize}

\begin{theorem}[Sobolev-trace inequalities for piecewise $\W^{1,\p}$ functions: \textbf{case 1}] \label{theorem:sob-trace-pnatural}
Given~$\p$ in $[1,\infty)$ and $\q$ in $[1,\pnatural]$,
we have
\begin{equation}\label{eq:sob-theorem-1}
\Norm{v}_{\L^{\q}(\partial\Omega)}
\le \CSnatural(\p,\q,\Omega) \Norm{v}_{\W^{1,\p;\p}_{\GammaD}(\Tauh)}
\qquad\qquad\qquad
\forall v \in \W^{1,\p}_{\GammaD}(\Tauh), 
\end{equation}
where, given~$\CSP$ as in~\eqref{eq:broken-Sobolev} and
$C_1$ as in~\eqref{C1},
$$
\CSnatural(\p,\q,\Omega)
 := |\partial\Omega|^{\frac{\pnatural -\q}{\pnatural\q}} 
        C_1^{\frac{1}{\pnatural}}(\pnatural,\Omega)
        \CSP^{\frac{1}{\pnaturalprime}}\Big[
        \vert\Omega\vert^{\frac{(\postar-\pnatural)}{\postar\pnaturalprime}}
        + \Big( \frac{d}{\gamma} \CSP
        \vert\Omega\vert^{\frac{\postar}{\postar-\pddag}}
        + \frac{\pddag}{\gamma}
        \max_{\E\in\Tauh} (\hE^\p)
        \Big)^{\frac{1}{\pprime}}
        \Big]^{\frac{1}{\pnatural}}.
$$
\end{theorem}
\begin{proof}
We select $\q=\pnatural$ and $\s=\p$ in Proposition~\ref{prop:sob-trace-general} and observe that $\pprime(\pnatural-1)=\pddag$ to obtain
\begin{equation} \label{eq:start-case1}
\Norm{v}^{\pnatural}_{\L^{\pnatural}(\partial\Omega)} \le 
  C_1(\pnatural,\Omega)\left[
    \Norm{v}^{\pnatural-1}_{\L^{\pddag}(\Omega)}
    + \Big(\sum_{\E\in\Tauh}
    \hE
    \Norm{v}^{\pddag}_{\L^{\pddag}(\partial\E)}
    \Big)^{\frac{1}{\pprime}}\right]
    \Norm{v}_{\W^{1,\p;\p}_{\GammaD}(\Tauh)}
\end{equation}
Applying~\eqref{intermediate-trace} with $\q=\pddag$,
and observing
that $\pprime(\pddag-1)=\postar$ and
$1-\pddagprime/\pprime=1/(d\p)$, we infer
$$
\hE \Norm{v}_{\L^{\pddag}(\partial\E)}^{\pddag}
\le \frac{d}{\gamma}
    \vert \E\vert^{\frac{1}{d\p}}
    \Norm{v}^{\pddag}_{\L^{\postar} (\E)}
    + \frac{\pddag}{\gamma}\hE
        \Norm{\nabla v}_{\Lbf^\p(\E)}
        \Norm{v}_{\L^{\postar}(\E)}^{\pddag-1}.
$$
Using H\"older's inequality for sequences ($\postar/ (\postar-\pddag)$ and $\postar/ \pddag$), we remark that
\[
\frac{d}{\gamma}
\sum_{\E\in\Tauh} \vert\E\vert^{\frac1{d\p}} 
    \Norm{v}_{\L^{\postar}(\E)}
    ^{\pddag}
\le \frac{d}{\gamma} \Norm{v}_{\L^{\postar}(\Omega)}^{\pddag}
    (\sum_{\E\in\Tauh} \vert\E\vert)
    ^{\frac{\postar}{\postar-\pddag}}
=   \frac{d}{\gamma} 
    \vert\Omega\vert^{\frac{\postar}{\postar-\pddag}}
    \Norm{v}_{\L^{\postar}(\Omega)}^{\pddag}.
\]
Thus, using the previous bounds and again H\"older's inequality
for sequences ($\p$ and $\pprime$), we have
\[
\begin{split}
 \Big(\sum_{\E\in\Tauh}
    & \hE \Norm{v}^{\pddag}_{\L^{\pddag}(\partial\E)}
    \Big)^{\frac{1}{\pprime}}
\le \Big( \frac{d}{\gamma}\sum_{\E\in\Tauh} 
    \vert \E\vert^{\frac{1}{d\p}}
    \Norm{v}^{\pddag}_{\L^{\postar} (\E)}
    + \frac{\pddag}{\gamma}  
    \sum_{\E\in\Tauh}
\hE  \Norm{\nabla v}_{\Lbf^\p(\E)}
\Norm{v}_{\L^{\postar}(\E)}^{\pddag-1}
\Big)^{\frac{1}{\pprime}} \\
& \le 
\Big[ \frac{d}{\gamma}
    \vert\Omega\vert^{\frac{\postar}{\postar-\pddag}}
    \Norm{v}_{\L^{\postar}(\Omega)}^{\pddag}
    + \frac{\pddag}{\gamma} \Big(   \sum_{\E\in\Tauh}
\hE^\p  \Norm{\nabla v}_{\Lbf^\p(\E)}^\p
\Big)^{\frac{1}{\p}}
\Big(   \sum_{\E\in\Tauh}
\Norm{v}_{\L^{\postar}(\E)}^{\postar}
\Big)^{\frac{1}{\pprime}} \Big]^{\frac{1}{\pprime}}\\
& \le 
\Big[ \frac{d}{\gamma}
    \vert\Omega\vert^{\frac{\postar}{\postar-\pddag}}
      \Norm{v}_{\L^{\postar}(\Omega)}^{\pddag}
    + \frac{\pddag}{\gamma}
\max_{\E\in\Tauh} (\hE^\p)
    \Norm{\nabla v}_{\Lbf^\p(\Omega)}
\Norm{v}_{\L^{\postar}(\Omega)}^{\pddag-1}
\Big]^{\frac{1}{\pprime}}.
\end{split}
\]
Finally, we apply~\eqref{eq:broken-Sobolev} to deduce
\[
\Big(\sum_{\E\in\Tauh}
    \hE \Norm{v}^{\pddag}_{\L^{\pddag}(\partial\E)}
    \Big)^{\frac{1}{\pprime}}
\le 
\Big( \frac{d}{\gamma} \CSP^{\pddag}
    \vert\Omega\vert^{\frac{\postar}{\postar-\pddag}}
    + \frac{\pddag}{\gamma}
    \max_{\E\in\Tauh} (\hE^\p)
    \CSP^{\pddag-1} \Big)^{\frac{1}{\pprime}}
    \Norm{v}_{\W^{1,\p;\p}(\Tauh)}^{\frac{\pddag}{\pprime}}
\]
and
\[
\Norm{v}^{\pnatural-1}_{\L^{\pnatural}(\Omega)}
\le \vert\Omega\vert^{\frac{(\postar-\pnatural)}{\postar\pnaturalprime}}
    \Norm{v}_{\L^{\postar}(\Omega)}^{\pnatural-1}
\le \vert\Omega\vert^{\frac{(\postar-\pnatural)}{\postar\pnaturalprime}}
    \CSP^{\pnatural-1}
    \Norm{v}_{\W^{1,\p;\p}(\Omega)}^{\pnatural-1}.
\]
We insert these inequalities in~\eqref{eq:start-case1},
recall again that $\pddag/\pprime=\pnatural-1$,
and obtain
\[
\begin{split}
\Norm{v}^{\pnatural}_{\L^{\pnatural}(\partial\Omega)}
\le     C_1 \Big[
        \vert\Omega\vert^{\frac{(\postar-\pnatural)}{\postar\pnaturalprime}}
        \CSP^{\pnatural-1}
        + \Big( \frac{d}{\gamma} \CSP^{\pddag}
    \vert\Omega\vert^{\frac{\postar}{\postar-\pddag}}
    + \frac{\pddag}{\gamma}
    \max_{\E\in\Tauh} (\hE^\p)
    \CSP^{\pddag-1} \Big)^{\frac{1}{\pprime}}
    \Big]
    \Norm{v}_{\W^{1,\p;\p}(\Omega)}^{\pnatural}.
\end{split}
\]
The Sobolev trace inequality \eqref{eq:sob-theorem-1}
follows by taking the $\pnatural$-th root on both sides and
from
H\"older's inequality ($\pnatural/q$ and $(\pnatural/q)'$):
\[
\Norm{v}_{\L^{\q}(\partial\Omega)}\le |\partial\Omega|^{\frac{\pnatural -\q}{\pnatural\q}}
\Norm{v}_{\L^{\pnatural}(\partial\Omega)}.
\]
\end{proof}

\begin{theorem}[Sobolev-trace inequalities for piecewise $\W^{1,\p}$ functions: \textbf{case 2}] \label{theorem:sob-trace-psharp}
Given~$\p$ in $[1,d)$, we have
\begin{align}
\Norm{v}_{\L^{\q}(\partial\Omega)}
& \le \CSddag(\p,\q,\Omega) \Norm{v}_{\W^{1,\p;\pddag}_{\GammaD}(\Tauh)}
\qquad\quad
\forall v \in \W^{1,\p}_{\GammaD}(\Tauh),
\qquad \text{for } \q \in [\p, \pddag],
\label{eq:sob-theorem-2a}\\
\Norm{v}_{\L^{\q}(\partial\Omega)} 
& \le \CSsharp(\p,\q,\Omega) \Norm{v}_{\W^{1,\p;\psharp}_{\GammaD}(\Tauh)}
\qquad\quad
\forall v \in \W^{1,\p}_{\GammaD}(\Tauh),
\qquad \text{for } \q\in[\pddag, \psharp],
     \label{eq:sob-theorem-2b}
\end{align}
where, given~$\CSP$, $\CSPsharp$, and~$C_1$
as in~\eqref{eq:broken-Sobolev},
\eqref{eq:broken-Sobolev-sharp}, and~\eqref{C1},
and $C_4$ as in~\eqref{C4} below,
\begin{align*}
\CSddag(\p,\q,\Omega)
&   := |\partial\Omega|^{\frac{\pddag -\q}{\pddag\q}} 
    \frac{C_1^{\frac{1}{\pddag}}(\pddag,\Omega)}{\gamma^{\frac{d(\p-1)}{\pddag(d\p-1)}}}\Big(
    \left(\gamma^{\frac{d(\p-1)}{d\p-1}}+C_4(\p,\pddag)\right) (\CSP)^{\pddag-1} 
    + 
    \max_{\E\in\Tauh} \Big(\hE^{\frac{d(\p-1)^2}{\p(d-1)}}\Big) \Big)^{\frac{1}{\pddag}},\\
\CSsharp(\p,\q,\Omega)
&   := |\partial\Omega|^{\frac{\psharp -\q}{\psharp\q}} 
    \frac{C_1^{\frac{1}{\psharp}}(\psharp,\Omega)}{\gamma^{\frac{d(\p-1)}{\psharp \p(d-1)}}}\Big(
    \left(\gamma^{\frac{d(\p-1)}{\p(d-1)}}+C_4(\p,\psharp)\right) (\CSPsharp)^{\psharp-1} 
    + 1\Big)^{\frac{1}{\psharp}}.
\end{align*}
\end{theorem}
\begin{proof}
Taking $s=q$ in Proposition \ref{prop:sob-trace-general} yields
\begin{equation}\label{eq:start-case2}
\Norm{v}^{\q}_{\L^{\q}(\partial\Omega)} \le 
  C_1(\q,\Omega)\left[
    \Norm{v}^{\q-1}_{\L^{\pprime(\q-1)}(\Omega)}
    + \Big(\sum_{\E\in\Tauh}
    \hE^{d - \frac{(d-1)\pprime}{\qprime}}
    \Norm{v}^{\pprime(\q-1)}_{\L^{\q}(\partial\E)}
    \Big)^{\frac{1}{\pprime}}\right]
    \Norm{v}_{\W^{1,\p;\q}_{\GammaD}(\Tauh)}.
\end{equation}
Given the unit ball~$\mathcal B_d$ in~$\Rbb^d$, we recall that
$
|\E|
\le \hE^d |\mathcal{B}_d|
=\hE^d \pi^{\frac d2}\Gamma \big( \frac d2 + 1 \big)^{-1}.
$
Using this bound followed by Young's inequality,
and introducing the constant $C_4 = C_4(\p,\q)$ defined as
\begin{equation}\label{C4}
C_4(\p,\q) := d^{\frac1{\qprime}} \left(\frac{\pi^{\frac d2}}
    {\Gamma \big( \frac d2 + 1 \big)}\right)^{\frac{1}{\qprime}-\frac{1}{\pprime}} + (q-1)^{\frac1{\qprime}},
\end{equation}
we deduce
\[
\begin{split}
&  \hE^{d - \frac{(d-1)\pprime}{\qprime}}
    \Norm{v}_{\L^{\q}(\partial\E)}^{\pprime(\q-1)}
    = \Big(\hE^{\frac{d\qprime}{\pprime} - d+1}
    \Norm{v}_{\L^{\q}(\partial\E)}^\q\Big)^{\frac{\pprime}{\qprime}}\\
&   \overset{\eqref{intermediate-trace}}{\le}
    \Big(\frac{d}{\gamma} 
    \hE^{d\big(\frac{\qprime}{\pprime}-1\big)}
    \vert \E\vert^{1-\frac{\qprime}{\pprime}}
    \Norm{v}^\q_{\L^{\pprime(\q-1)} (\E)}
    + \frac{\q}{\gamma} \hE^{\frac{d\qprime}{\pprime} - d+1}
        \Norm{\nabla v}_{\Lbf^\p(\E)}
        \Norm{v}_{\L^{\pprime(\q-1)}(\E)}^{\q-1}\Big)^{\frac{\pprime}{\qprime}}\\
&   \le \Big(  \frac{d}{\gamma} 
    \Big(\frac{\pi^{\frac d2}}
    {\Gamma \big( \frac d2 + 1 \big)}\Big)
    ^{1-\frac{\qprime}{\pprime}}
    \Norm{v}^\q_{\L^{\pprime(\q-1)} (\E)}
    + \frac{\q}{\gamma}
    \hE^{\frac{d\qprime}{\pprime} - d+1}
        \Norm{\nabla v}_{\Lbf^\p(\E)}
        \Norm{v}_{\L^{\pprime(\q-1)}(\E)}^{\q-1}\Big)^{\frac{\pprime}{\qprime}}\\
&   \le  \Big[ \frac1\gamma\Big(
    d \Big(\frac{\pi^{\frac d2}}
    {\Gamma \big( \frac d2 + 1 \big)}\Big)
    ^{1-\frac{\qprime}{\pprime}}
    +q-1 \Big)
    \Norm{v}^\q_{\L^{\pprime(\q-1)} (\E)}
    + 
    \frac{\hE^{\q-d\q+d\q\frac{\qprime}{\pprime}}}\gamma
        \Norm{\nabla v}_{\Lbf^\p(\E)}^\q 
        \Big]^{\frac{\pprime}{\qprime}}\\
&   \le \gamma^{-\frac{\pprime}{\qprime}} \left(
    C_4^{\pprime}
    \Norm{v}^{\pprime(\q-1)}_{\L^{\pprime(\q-1)} (\E)}
    + \hE^{(\q-1)(\pprime+d\qprime-d\pprime)} 
        \Norm{\nabla v}_{\Lbf^\p(\E)}^{\pprime(\q-1)}\right).
\end{split}
\]
A consequence of the previous bound is
\[
\begin{split}
\sum_{\E\in\Tauh}
    \hE^{d - \frac{(d-1)\pprime}{\qprime}}
    \Norm{v}^{\pprime(\q-1)}_{\L^{\q}(\partial\E)}
&   \le 
    \gamma^{-\frac{\pprime}{\qprime}}\sum_{\E\in\Tauh}\Big( 
    C_4^{\pprime}\Norm{v}^{\pprime(\q-1)}_{\L^{\pprime(\q-1)} (\E)}
    + \hE^{(\q-1)(\pprime+d\qprime-d\pprime)}
        \Norm{\nabla v}_{\Lbf^\p(\E)}^{\pprime(\q-1)}\Big) \\
& \le  
    \gamma^{-\frac{\pprime}{\qprime}} 
    \Big(C_4^{\pprime} 
    \Norm{v}_{\L^{\pprime(\q-1)}(\Omega)}^{\pprime(\q-1)}
    + \max_{\E\in\Tauh} (\hE^{(\q-1)(\pprime+d\qprime-d\pprime)})
    \Norm{\nablah v}_{\Lbf^\p(\Omega)}^{\pprime(\q-1)}
    \Big).
\end{split}
\]
Plugging the previous result into \eqref{eq:start-case2} implies
\[
\begin{split}
& C_1^{-1}
    \Norm{v}^{\q}_{\L^{\q}(\partial\Omega)}\\
& \le  \left[
    \Norm{v}^{\q-1}_{\L^{\pprime(\q-1)}(\Omega)}
    + \gamma^{-\frac{1}{\qprime}} 
    \Big(C_4 
    \Norm{v}_{\L^{\pprime(\q-1)}(\Omega)}^{(\q-1)}
    + \max_{\E\in\Tauh} \Big(\hE^{(\q-1)(1-d+\frac{d\qprime}{\pprime})}\Big)
    \Norm{\nablah v}_{\Lbf^\p(\Omega)}^{(\q-1)}
    \Big)\right]
    \Norm{v}_{\W^{1,\p;\q}_{\GammaD}(\Tauh)} \\
& \le \frac{\gamma^{\frac{1}{\qprime}} + C_4}{\gamma^{\frac{1}{\qprime}}}
    \Norm{v}^{\q-1}_{\L^{\pprime(\q-1)}(\Omega)}
    \Norm{v}_{\W^{1,\p;\q}_{\GammaD}(\Tauh)}
    + \frac1{\gamma^{\frac{1}{\qprime}}}
    \max_{\E\in\Tauh} \Big(\hE^{(\q-1)(1-d+\frac{d\qprime}{\pprime})}\Big)
    \Norm{v}_{\W^{1,\p;\q}_{\GammaD}(\Tauh)}^{\q}.
\end{split}
\]
Taking~$\q=\pddag$,
and using the broken Sobolev-Poincar\'e
inequality~\eqref{eq:broken-Sobolev} together with
bound~\eqref{eq:trivial-bound}
and the identity
\[
(\pddag-1)
\Big(1-d+d \frac{\pddagprime}{\pprime} \Big) 
= \frac{d(\p-1)^2}{\p(d-1)},
\]
we get
\[
\begin{split}
\Norm{v}^{\pddag}_{\L^{\pddag}(\partial\Omega)}
&   \le \frac{C_1}{\gamma^{\frac{d(\p-1)}{d\p-1}}}\Big(
    \left(\gamma^{\frac{d(\p-1)}{d\p-1}}+C_4\right) (\CSP)^{\pddag-1} 
    + 
    \max_{\E\in\Tauh} \Big(\hE^{\frac{d(\p-1)^2}{\p(d-1)}}\Big) \Big)
    \Norm{v}_{\W^{1,\p;\pddag}_{\GammaD}(\Tauh)}^{\pddag}.
\end{split}
\]
If instead we take $\q=\psharp$,
use the broken Sobolev-Poincar\'e
inequality~\eqref{eq:broken-Sobolev-sharp},
and observe
\[
(\psharp-1)
\Big(1-d+d \frac{\psharpprime}{\pprime} \Big) = 0,
\]
then we obtain
\[
\begin{split}
\Norm{v}^{\psharp}_{\L^{\psharp}(\partial\Omega)}
&   \le \frac{C_1}{\gamma^{\frac{d(\p-1)}{\p(d-1)}}}\Big(
    \left(\gamma^{\frac{d(\p-1)}{\p(d-1)}}+C_4\right)
    (\CSPsharp)^{\psharp-1} + 1 \Big)
    \Norm{v}_{\W^{1,\p;\psharp}_{\GammaD}(\Tauh)}^{\psharp}.
\end{split}
\]
Bounds~\eqref{eq:sob-theorem-2a}
and~\eqref{eq:sob-theorem-2b}
follow using H\"older's inequality ($r/\q$ and $(r/\q)'$)
with $r=\pddag$ and $r=\psharp$.
\end{proof}

\begin{theorem}[Sobolev-trace inequalities for piecewise $\W^{1,\p}$ functions: \textbf{case 3}] \label{theorem:sob-trace-3}
Given~$\p$ in $[d,\infty)$ and $\q$ in $[\p,\infty)$,
we have
\begin{equation} \label{eq:sob-theorem-3}
\Norm{v}_{\L^{\q}(\partial\Omega)} 
 \le \CSflat(\p,\q,\Omega) \Norm{v}_{\W^{1,\p;\q}_{\GammaD}(\Tauh)}
 \qquad\qquad\qquad
  \forall v \in \W^{1,\p}_{\GammaD}(\Tauh),
\end{equation}
where, given~$\CSsharp$
as in~\eqref{eq:sob-theorem-2b},
\begin{align*}
\CSflat(\p,\q,\Omega)
&   := \CSsharp(\qflat,\q,\Omega)
    \max_{\E\in\Tauh}\big( \vert\E\vert^{-1}
    \hE^{d} \card{(\FcalE)} \big)
    ^{\frac{\p-\qflat}{\p\qflat}}
    \vert\Omega\vert^{\frac{\p-\qflat}{\p\qflat}}.
\end{align*}
\end{theorem}
\begin{proof}
Recall that $\qflat$ in $(1,d)$ is such that,
cf. \eqref{indices}--\eqref{relation-indices},
\begin{equation} \label{relation:r-d-p}
\qflat:=\frac{\q d}{\q +d -1}<d \le \p.
\end{equation}
Using~\eqref{eq:sob-theorem-2b}, we write
\[
\Norm{v}_{\L^{\q}(\partial\Omega)} 
\le \CSsharp(\qflat,\q,\Omega)
\Norm{v}_{\W^{1,\qflat;\q}_{\GammaD}(\Tauh)}.
\]
To conclude, we need to estimate from above
the two terms in the DG norm
appearing on the right-hand side;
cf.~\eqref{norm-broken}.
We begin with the gradient term:
\[
\Norm{\nablah v}_{\Lbf^{\qflat}(\Omega)}
\overset{\eqref{relation:r-d-p}}{\le} 
\vert \Omega \vert^{\frac{\p-\qflat}{\p\qflat}}
\Norm{\nablah v}_{\Lbf^\p(\Omega)}.
\]
We now focus on the jump terms: applying H\"older's inequality for sequences ($\p/(\p-\qflat)$ and $\p/\qflat$), we obtain
\[
\begin{split}
 \big( \sum_{\F\in\FcalhID}& \Norm{\jump{v}}_{\L^\q(\F)}^{\qflat} \big)^\frac1{\qflat}
= \big( \sum_{\F\in\FcalhID} \htildeF^{\frac{d(\p-\qflat)}{\p}}
    \htildeF^{-\frac{d(\p-\qflat)}{\p}} \Norm{\jump{v}}_{\L^{\qflat}(\F)}^{\qflat} \big)^\frac1\r \\
&  \le 
\big( \sum_{\F\in\FcalhID} \htildeF^{d} \big)
    ^{\frac{\p-\qflat}{\p\qflat}}
\big( \sum_{\F\in\FcalhID}
    \htildeF^{-\frac{d(\p-\qflat)}{\qflat}}
    \Norm{\jump{v}}_{\L^\q(\F)}^\p \big)^\frac1\p\\
& \le 
\big( \sum_{\E\in\Tauh}
    \sum_{\F\in\FcalE}\htildeF^{d} \big)
    ^{\frac{\p-\qflat}{\p\qflat}}
\big( \sum_{\F\in\FcalhID}
    \htildeF^{-\frac{d(\p-\qflat)}{\qflat}}
    \Norm{\jump{v}}_{\L^\q(\F)}^\p \big)^\frac1\p \\
& \le 
\max_{\E\in\Tauh}\big( \vert\E\vert^{-1}
    \hE^{d} \card{(\FcalE)} \big)
    ^{\frac{\p-\qflat}{\p\qflat}}
    \vert\Omega\vert^{\frac{\p-\qflat}{\p\qflat}}
\big( \sum_{\F\in\FcalhID}
    \htildeF^{-\frac{d(\p-\qflat)}{\qflat}}
    \Norm{\jump{v}}_{\L^\q(\F)}^\p \big)^\frac1\p.
\end{split}
\]
The assertion follows noting that
\[
-\frac{d(\p-\qflat)}{\qflat}
=
-\frac{\p}{\qprime}
    -\frac{d\p}{\pprime}
    +\frac{d\p}{\qprime}.
\]
\end{proof}

\begin{remark}[Applicability of
trace inequalities for piecewise $\W^{1,\p}$ functions
in the analysis of nonconforming Galerkin methods]
\label{remark:where-to-use}
Inequality~\eqref{eq:sob-theorem-1}
is suited for the analysis of interior penalty
discontinuous Galerkin methods
for nonlinear problems; such methods in fact
involve the $\Lbf^\p$ norm of the broken gradient
and the $\L^\p$ norm of the jumps.
For $\p<d$, inequality~\eqref{eq:sob-theorem-2b}
exhibits the maximal Lebesgue regularity
indices on the left-hand side
and on the jump terms,
while inequality~\eqref{eq:sob-theorem-2a}
is its counterpart with weaker norms;
a variant of both inequalities (namely that with a projection
onto constants over facets of the jumps)
can be employed in the analysis of Crouzeix-Raviart schemes,
as it allows for improved Lebesgue norms
on the left-hand side compared to~\eqref{eq:sob-theorem-1}.
Inequality~\eqref{eq:sob-theorem-3}
is instead concerned with the case~$\p\ge d$
and involves the same Lebesgue norms on $\partial\Omega$
and on the facet jumps.
\eremk
\end{remark}

\paragraph*{Acknowledgments.}
MB and LM have been partially funded by the
European Union (ERC, NEMESIS, project number 101115663);
views and opinions expressed are however those
of the authors only and do not necessarily reflect
those of the EU or the ERC Executive Agency.
LM has been partially funded by MUR (PRIN2022 research grant n. 202292JW3F). MB has been partially supported by the INdAM-GNCS project CUP E53C24001950001.
The authors are members of the Gruppo Nazionale Calcolo Scientifico-Istituto Nazionale di Alta Matematica (GNCS-INdAM).

{\footnotesize
\bibliography{bibliogr.bib}}
\bibliographystyle{plain}

\end{document}